\documentclass[11pt]{article}
\usepackage{amsmath}
\usepackage{amssymb}
\usepackage{amsfonts}
\usepackage{eucal}
\newtheorem{theorem}{Theorem}[section]
\newtheorem{proposition}[theorem]{Proposition}

\newtheorem{corollary}[theorem]{Corollary}

\newtheorem{lemma}[theorem]{Lemma}
\newtheorem{remark}[theorem]{Remark}

\newenvironment{proof}[1][Proof]{\textbf{#1.} }{\ \rule{0.5em}{0.5em}}
\input pictex.sty

\usepackage{xcolor}

\usepackage{soul}

\DeclareMathOperator\Var{Var}
\DeclareMathOperator\Cov{Cov}

\begin{document}

 \title{On exact laws of large numbers for Oppenheim expansions with infinite mean}%\thanks{The support of Gruppo Nazionale per l'Analisi Matematica, la Probabilit\`{a} e le loro Applicazioni (GNAMPA) of the Istituto Nazionale di Alta Matematica (INDAM) is  acknowledged.} 
\author{Rita Giuliano \footnote{Dipartimento di
		Matematica, Universit\`a di Pisa, Largo Bruno
		Pontecorvo 5, I-56127 Pisa, Italy (email: rita.giuliano@unipi.it)}~~and Milto Hadjikyriakou\footnote{School of Sciences, University of Central Lancashire, Cyprus campus, 12-14 University Avenue, Pyla, 7080 Larnaka, Cyprus (email:
		mhadjikyriakou@uclan.ac.uk).}} 	
\maketitle	

\begin{abstract}
	
	\noindent
	In this work we investigate the asymptotic behaviour of  weighted partial sums of a particular class of random variables related to Oppenheim series expansions.  More precisely, we verify convergence in probability as well as almost sure convergence to a strictly positive and finite constant without assuming any dependence structure or the existence of means. Results of this kind are known as  {\it exact weak} and {\it exact strong} laws. 
\end{abstract}


\section{Introduction}
Consider a sequence $\{X_n, n\geq 1\}$ with independent and identically distributed random variables.  If the random variables have nonzero finite mean, Kolmogorov's strong law of large numbers implies that 
\[
\lim_{n\to\infty}\frac{1}{n\mu}\sum_{k = 1}^{n}X_k= 1\quad\mbox{a.s.}
\]
where $\mu$ denotes the common mean of the random variables. It has been proven that in the case of zero mean or in the case where the mean does not exist, such a strong law is not valid (see for example \cite{M1978} and \cite{CR1961}). However, similar asymptotic results can be obtained in some cases by correctly adjusting the weights involved. Similar peculiar cases can be found in the literature of weak laws. In fact,   it was proven in \cite{F1971} that, for $\{X_n, n\geq 1\}$ independent and identically distributed random variables with $S_n = \sum_{i=1}^{n}X_i$,
\[
\frac{S_n -nEX_1I\{|X_1|\leq n\}}{n}\to 0\quad\mbox{in probability as}\quad n\to \infty
\]
if and only if 
\[
xP(|X_1|>x)\to 0\quad\mbox{as}\quad x\to \infty.
\]
The above result implies that the condition of the existence of means is not necessary for obtaining a weak law of large numbers. Typical examples of this case are the well-known St. Petersburg game described in \cite{F1968} and Feller game presented in \cite{MN2013}. 

Thus, it is important to study {\it weighted} laws of large numbers i.e. to identify sequences of real numbers $(a_n)_{n\geq 1}$ and $(b_n)_{n\geq 1}$ such that $\frac{\sum_{k = 1}^{n}a_k X_k}{b_n}$ converges to 1 either in probability or almost surely. These kind of problems are called \textit{exact weak and exact strong laws of large numbers} respectively. 

The case of exact strong laws has been studied extensively by Adler (see \cite{A2000} and all the references therein), while in \cite{HP1993} and \cite{AM2018} the assumption of independence has been relaxed. Exact weak laws for i.i.d. random variables can be found in \cite{A2007}, \cite{A2008} and \cite{N2016}, while the assumption of identically distributed random variables is dropped in \cite{A2012}. Exact weak laws of large numbers can also be found in the literature for dependent random variables (see for example \cite{MA2018} and \cite{W2019}).

Throughout the paper, the notation $a_n\sim b_n$, $a_n = o(b_n)$ and $f(x)\asymp g(x)$ will be used to denote \[\lim_{n\to \infty}\frac{a_n}{b_n}=1,\quad \lim_{n\to \infty}\frac{a_n}{b_n}=0\quad\mbox{and}\quad 0<\liminf_{x\to 0}\frac{f(x)}{g(x)}\leq \limsup_{x\to 0}\frac{f(x)}{g(x)}<\infty \]  
respectively while the constant $C$ will be used to denote a real number that is not necessarily the same in every appearance. We use the convention $\sum_a^b=0$ if $b<a$, while $\lceil x \rceil$ is used to denote the least integer greater than or equal to $x$. Last, by the symbol $\mathbb{N}^*$ we mean the set of integers $\{1, 2, 3, \dots\}$ and the symbol $I(A)$ denotes the indicator function of the set $A$.

We are interested in obtaining weighted weak and strong laws of large numbers for a particular class of random variables related to Oppenheim expansions. The framework of our work is described below.

Let $(B_n)_{n\geq  1}$ be a sequence of integer valued random variables defined on $(\Omega, \mathcal{A}, P)$, where $\Omega =[0,1]$, $\mathcal{A}$ is the $\sigma$-algebra of the Borel subsets of $[0,1]$ and $P$ is the Lebesgue measure on $[0,1]$. Let $\{F_n, n\geq 1\}$ be a sequence of probability distribution functions defined on $[0,1]$ with  $F_n(0)=0$, $\forall n$ and moreover let $\varphi_n:\mathbb{N}^*\to \mathbb{R}^+$ be a sequence of functions. Furthermore, let $(y_n)_{n\geq  1}$ with $y_n=y_n(h_1, \dots, h_n)$ be a sequence of nonnegative numbers (i.e. possibly depending on the $n$ integers $h_1, \dots, h_n$) such that, for $h_1 \geq  1$ and $h_j\geq  \varphi_{j-1}(h_{j-1})$, $j=2, \dots, n$ we have

\[\label{densitacondizionale}
P\big(B_{n+1}=h_{n+1}|B_{n}=h_{n}, \dots, B_{1}=h_{1}\big)= F_n(\beta_n)-F_n(\alpha_n),
\]
where 
\[
\alpha_n=\delta_n(h_n, h_{n+1}+1, y_n)  ,\quad \beta_n=\delta_n(h_n, h_{n+1}, y_n)\quad\mbox{with}\quad\delta_j(h,k, y) = \frac{ \varphi_j (h )(1+y )}{k+\varphi_j (h ) y  }.
\]
Let $Y_n= y_n(B_1, \dots, B_n)$ and define
\begin{equation}
\label{Rdef}R_{n}= \frac{ B_{n+1}+\varphi_n(B_n) Y_n}{\varphi_n(B_n)(1+Y_n) }= \frac{1}{\delta_n(B_n, B_{n+1}, Y_n)}.
\end{equation}
Particular instances of this scheme are studied in \cite{L1883}, \cite{G1976} (L\"uroth series), \cite{S1974}, \cite{E1913} (Engel series), \cite{P1960} (Sylvester series), \cite{HKS2002}  (Engel continued fraction expansions). Recently, in \cite{G2018} the convergence of
\[
\frac{1}{n\log n}\sum_{k = 1}^{n} R_k
\]
was studied and a weak law of large number was obtained (see Theorem 2.2 there). 

The purpose of the present work is to obtain exact laws for the random variables $(R_n)_{n\geq 1}$, i.e. to find suitable sequences of real numbers $(a_n)_{n\geq 1}$ and $(b_n)_{n\geq 1}$ such that the convergence of
\[
\frac{1}{b_n}\sum_{k = 1}^{n}a_k R_k
\]
to a positive finite number is established either in probability or almost surely. The paper is structured as follows. In Section 2 we present some preliminary results that are instrumental for obtaining the main results of this work. In Section 3 we present some exact weak laws while the last section of the paper is devoted to exact strong laws.

\section{Preliminaries}
First observe that for every $n$ and  for every fixed $h$ and $y$, we have $\delta_n(h ,\varphi (h), y)=1$, hence
\begin{equation}\label{intevallo}
\bigcup_{k\geq \varphi_n (h)}\big[\delta_n(h,k+1,y),\delta_n(h,k,y)\big]= \lim_{k\to \infty}\left[\frac{ \varphi_n (h )(1+y )}{k+1+\varphi_n (h ) y  },1\right]=(0,1],
\end{equation}
so that
\[
\sum_{k\geq \varphi_n (h)} \int_{\delta_n(h,k+1,y)}^{\delta_n(h,k,y)} \,{\rm d  }{F_j}(u)= \int_0^1  \,{\rm d }{F_j}(u)=1.
\]
For every integer $n$, let $U_n$ be a  random variable with distribution $F_n$. Then the characteristic function of $Y_n:= \frac{1}{U_n}$ is
$$\psi_n(t) =\int_0^1 {\rm e}^{{\rm i}\frac{t}{u}}  \,{\rm d  }{F_n}(u).$$
Furthermore, notice that for every $n$ and for every fixed $h$ and $y$, relation (\ref{intevallo}) allows us to write the characteristic of $Y_n$ in the following form 
\[
\psi_n(t) =\sum_{k\geq \varphi_n (h)} \int_{\delta_n(h,k+1,y)}^{\delta_n(h,k,y)} {\rm e}^{{\rm i}\frac{t}{u}}\,{\rm d  }{F_n}(u).
\]
We start by stating two known results that are important tools for obtaining Theorem \ref{teoremadistanza}. Although the original results stated in \cite{G2018}  concern  identical absolutely continuous distributions, the same results are valid even in our more general framework, where   the only assumption needed is the existence of the distribution functions $F_n$. The proofs are omitted for brevity. 

\begin{lemma}[\cite{G2018}, Lemma 4.1] \sl 
	\label{lemmadistanza}
	Let the integer $h$ and the positive number $y$ be fixed. Then, for every $t \in \mathbb{R}$ and for every integer $n \geq 1$,
	$$\left|\sum_{k \geq \varphi_n(h)}{\rm e}^{{\rm i}\frac{t}{\delta_n(h,k,y)}}\int_{\delta_n(h,k+1,y)}^{\delta_n(h,k,y)}{\rm d }{F_n}(x)-\psi_n(t)\right|\leq |t|.$$
\end{lemma}

\noindent The case $n=1$ of  Lemma \ref{lemmadistanza} is isolated for future reference in the corollary that follows.

\begin{corollary}[\cite{G2018}, Corollary 4.2]
	\label{cason=1} Let $\phi_{R_1}$ be the characteristic function of $R_1$. Then, for  every $t \in \mathbb{R}$,
	$$\left|\phi_{R_1}(t)-\psi_1(t)\right|\leq |t|.$$
\end{corollary}
Lemma \ref{lemmadistanza} and Corollary \ref{cason=1} are instrumental for obtaining the result that follows. 

\begin{theorem}\sl \label{teoremadistanza}  Let $(R_n)_{n\geq 1}$ be as in (\ref{Rdef}) and let $U_1,\ldots, U_n$ be independent random variables such that $U_n\sim F_n$ for any integer $n$. Let $\phi_{R_1, \dots, R_n}$ be the characteristic function of the vector $(R_1, \dots, R_n)$ and let $\psi_n$ be the characteristic function of the random variable defined as $Y_n = U_n^{-1}$ for every $n$. Then, for every $  (t_1, \dots, t_n)\in \mathbb{R}^n$ and $n \geq 1$ we have
	$$\left| \phi_{R_1, \dots, R_n}(t_1, \dots, t_n)-\prod_{k=1}^n\psi_k(t_k)\right|\leq  \sum_{k=1}^n |t_k|.$$
\end{theorem}
\begin{proof}
	By Corollary \ref{cason=1}, it suffices to show that, for every $n \geq 1$, we have 
	\begin{equation}\label{formaalternativa1}
	\left| \phi_{R_1, \dots, R_n}(t_1, \dots, t_n)-\prod_{k=1}^n\psi_k(t_k)\right|\leq \sum_{k=2}^n |t_k| +\Big|\phi_{R_1}(t_1)-\psi_1(t_1)\Big|.
	\end{equation}
	With the case $n=1$ being obvious, we can assume $n \geq 2$.
	For simplicity, let $y_k:= y_k(h_1, \dots, h_k)$ and
	\begin{equation}\label{posizioni}
	r_k:= r_{k}(h_1, \dots, h_{k+1})= \frac{ h_{k+1}+\varphi_k(h_k)y_k(h_1, \dots, h_k)}{\varphi_k(h_k)(1+y_k(h_1, \dots, h_k))}
	=\frac{1}{\delta_k(h_k, h_{k+1}, y_k)}.
	\end{equation}
	First we write the characteristic function $\phi_{R_1, \dots, R_n}$ in a suitable form. Note that the subscript $R_1, \dots, R_n$ is eliminated for simplicity. For every $n \geq 2$ put
	$$\mathcal{E}_n:=\big\{  (h_1, \dots, h_n )\in {N}^*:h_1\geq 1, h_i\geq \varphi_{i-1}(h_{i-1})\hbox{ \rm for every } i=2, \dots, n \big\}$$
and let 
$$\mathcal{B}_n := \{B_1=h_1, \dots, B_{n}=h_{n}\}.$$
	Then
	\begin{align*} &
	\phi(t_1, \dots, t_n) =E\left[{\rm e}^{{\rm i} \sum_{k=1}^n t_k R_k }\right]= \sum_{(h_1, \dots, h_{n+1} )\in\mathcal{E}_{n+1}}P\big(\mathcal{B}_{n+1}\big){\rm e}^{{\rm i} \sum_{k=1}^nt_kr_{k}}\\
	&=
	\sum_{(h_1, \dots, h_{n+1} )\in\mathcal{E}_{n+1}}P\big(  B_{n+1}=h_{n+1}|\mathcal{B}_n\big)P\big( \mathcal{B}_n\big){\rm e}^{{\rm i} \sum_{k=1}^nt_{k} r_{k}}\\
	&
	= \sum_{(h_1, \dots, h_{n+1} )\in\mathcal{E}_{n+1}} \left\{ \big(F_n(\beta_n) - F_n (\alpha_n)\big)  {\rm e}^{{\rm i}t_nr_{n}}\right\}P\big(  \mathcal{B}_n\big){\rm e}^{{\rm i} \sum_{k=1}^{n-1}t_kr_{k}}\\
	&
	=\sum_{(h_1, \dots, h_{n} )\in\mathcal{E}_{n}}P\big(  \mathcal{B}_n\big){\rm e}^{{\rm i} \sum_{k=1}^{n-1}t_kr_{k}}\left\{\sum_{\varphi_n(h_n)\leq h_{n+1}}\big( F_n(\beta_n) - F_n (\alpha_n)\big){\rm e}^{{\rm i}t_nr_{n}}\right\}.
	\end{align*}	
	Thus,
	\begin{align*} &
	\phi(t_1, \dots, t_n) -\prod_{k=1}^n\psi_k(t_k) \\
	&
	=\sum_{(h_1, \dots, h_{n} )\in\mathcal{E}_{n}}P\big(  \mathcal{B}_n\big){\rm e}^{{\rm i} \sum_{k=1}^{n-1}t_kr_{k}}\left\{\sum_{\varphi(h_n) \leq h_{n+1}}\big(F_n(\beta_n) - F_n (\alpha_n)  \big){\rm e}^{{\rm i}t_nr_{n}}-\psi_n (t_n)\right\}+\\
	&~~~~+\psi_n(t_n)\left\{\sum_{(h_1, \dots, h_{n} )\in\mathcal{E}_{n}} P\big( \mathcal{B}_n\big)
	{\rm e}^{{\rm i} \sum_{k=1}^{n-1}t_kr_{k}}- \prod_{k=1}^{n-1}\psi_k(t_k)\right\}\\
	&
	=\sum_{(h_1, \dots, h_{n} )\in\mathcal{E}_{n}}P\big( \mathcal{B}_n\big){\rm e}^{{\rm i} \sum_{k=1}^{n-1}t_{k}r_{k}}\left\{\sum_{\varphi_n(h_n) \leq h_{n+1}}\big( F_n(\beta_n)-F_n(\alpha_n)\big){\rm e}^{{\rm i}\frac{t_n}{\delta_n(h_n, h_{n+1}, y_n)}}-\psi_n(t_n)\right\}\\
	&~~~~+\psi_n(t_n)\left\{\sum_{(h_1, \dots, h_{n} )\in\mathcal{E}_{n}} P\big(  \mathcal{B}_n\big)
	{\rm e}^{{\rm i} \sum_{k=1}^{n-1}t_kr_{k}}- \prod_{k=1}^{n-1}  \psi_k(t_k)\right\},
	\end{align*}
	by the last equation in \eqref{posizioni}. Setting
	$$\Delta_n( t_1, \dots, t_n) = \left|\phi(t_1, \dots, t_n) -\prod_{k=1}^n\psi_k(t_k)  \right|,$$
	and using Lemma \ref{lemmadistanza} we have that 
	\begin{eqnarray*} 
	\Delta_n(t)&\leq& \sum_{(h_1, \dots, h_{n} )\in\mathcal{E}_{n}}P\big(  \mathcal{B}_n\big)\left|{\rm e}^{{\rm i} \sum_{k=1}^{n-1}t_kr_{k}}\right|\left| \sum_{\varphi_n(h_n) \leq h_{n+1}}\left(\int_{\alpha_n}^{\beta_n}{\rm d} {F_n} (u) \right){\rm e}^{{\rm i}\frac{t_n}{\delta(h_n, h_{n+1}, y_n)}} -\psi_n(t_n) \right|\\
	&~~~~+& \big|\psi_n(t_n)\big|\Delta_{n-1}(t_1, \dots, t_{n-1})\\
	&\leq&
	|t_n|  \sum_{(h_1, \dots, h_{n} )\in\mathcal{E}_{n}}P\big(  \mathcal{B}_n\big)+ \Delta_{n-1}(t_1, \dots, t_{n-1})\\
	&=&|t_n| +\Delta_{n-1}(t_1, \dots, t_{n-1}).
	\end{eqnarray*}
	Statement \eqref{formaalternativa1} follows immediately by induction.
	\end{proof}
	
	\begin{remark}
		Theorem \ref{teoremadistanza} can be considered as a generalization of Lemma 4.1 in \cite{G2018}.
	\end{remark}
	The results that follow allow us to provide upper and lower bounds for the quantities $P(R_i>x)$ and $P(R_i>x,R_j>y)$ for $x,y\geq 1$.
	 
	\begin{lemma} 
		\label{udist}Let $(R_n)_{n\geq 1}$ be as  in \eqref{Rdef}. Then, for any integer $n$ and for $x\geq 1$,
		\[
		E\left[F_n\left(\frac{\varphi_n(B_n)(1+Y_n)}{x\varphi_n(B_n)(1+Y_n)+1  }\right)\right] \leq P(R_n>x)\leq F_n\left(\frac{1}{x}\right).
		\]
	\end{lemma}
	
\begin{proof}
	Notice first that since $B_{n+1} \geq \varphi_n(B_n)$ we have that $R_n \geq 1$. We start with the calculation of $P(R_n > x)$, $x \geq 1$. By definition we can write, 
	$$P\big(R_n > x\big)= \sum_{(h_1, \dots, h_{n+1})\in \mathcal{E}_{n+1}}P\big(\mathcal{B}_{n+1} \big)I(r_n>x),$$ 
	where $r_n$ is as defined in \eqref{posizioni} and $\mathcal{B}_n := \{B_1=h_1, \dots, B_{n}=h_{n}\}$. Hence, the RHS of the latter expression can be written as
	\begin{eqnarray*}
	\sum_{(h_1, \dots, h_{n+1})\in \mathcal{E}_{n+1}}P\big(\mathcal{B}_{n+1} \big)I(r_n>x)
	&=&\sum_{(h_1, \dots, h_{n+1})\in \mathcal{E}_{n+1}}P\big(\mathcal{B}_n \big)P \big(B_{n+1}=h_{n+1}|\mathcal{B}_n \big )I(r_n>x)\\
	&=&\sum_{(h_1, \dots, h_{n})\in \mathcal{E}_{n}}P\big(\mathcal{B}_n \big)\sum_{h_{n+1} \geq \varphi_n(h_n)}P \big(B_{n+1}=h_{n+1}|\mathcal{B}_n \big )I(r_n>x)\\
	&=&\sum_{(h_1, \dots, h_{n})\in \mathcal{E}_{n}}P\big(\mathcal{B}_n \big)\sum_{h_{n+1} \geq \varphi_n(h_n)}\big\{F_n(\beta_n)-F_n(\alpha_n)\big\}I(r_n>x)\\
	&=&\sum_{(h_1, \dots, h_{n})\in \mathcal{E}_{n}}P\big(\mathcal{B}_n \big)\sum_{h_{n+1} \geq \varphi_n(h_n)\atop r_n > x}\big\{F_n(\beta_n)-F_n(\alpha_n)\big\}.
	\end{eqnarray*}
	Now $r_n > x$ if and only if
	$$h_{n+1}+\varphi_n(h_n)y_n > x\varphi_n(h_n)+ x y_n \varphi_n(h_n),$$
	or equivalently
	$$h_{n+1}> x\varphi_n(h_n)+ (x-1)y_n \varphi_n(h_n).$$
	Since
	$$x\varphi_n(h_n)+ (x-1)y_n \varphi_n(h_n)\geq x\varphi_n(h_n)\geq \varphi_n(h_n), $$
	the conditions under the inner sum become
	$$h_{n+1}> x\varphi_n(h_n)+ (x-1)y_n \varphi_n(h_n),$$
	or equivalently
	$$h_{n+1}\geq \lceil x\varphi_n(h_n)+ (x-1)y_n \varphi_n(h_n)\rceil=: s_n(x;h_1, \dots, h_{n} ).$$
	Hence,
	\begin{eqnarray}
\nonumber	P\big(R_n > x\big)&=& \sum_{ (h_1, \dots, h_{n})\in \mathcal{E}_{n}}P\big(\mathcal{B}_n \big)\sum_{h_{n+1} \geq s_n(x;h_1, \dots, h_{n} )  }\big\{F_n(\beta_n)-F_n(\alpha_n)\big\}\\
\nonumber &=& \sum_{ (h_1, \dots, h_{n})\in \mathcal{E}_{n}}P\big(\mathcal{B}_n \big)F_n\left(\frac{\varphi_n(h_n)(1+y_n)}{s_n(x;h_1, \dots, h_n) + \varphi_n(h_n) y_n}\right)
\\
&=&\label{eq1}E\left[  F_n\left(\frac{\varphi_n(B_n)(1+Y_n)}{S_n(x;B_1, \dots, B_n) + \varphi_n(B_n) Y_n}\right)\right ].
\end{eqnarray}
Notice that
\[
S_n(x;B_1, \dots, B_n) + \varphi_n(B_n) Y_n=\lceil x\varphi_n(B_n)+ (x-1)Y_n \varphi_n(B_n)\rceil+ \varphi_n(B_n) Y_n,
\]
so
\begin{align}
&\nonumber x\varphi_n(B_n)(1+Y_n)  =  x\varphi_n(B_n)+ (x-1)Y_n \varphi_n(B_n) + \varphi_n(B_n) Y_n \\&\leq\nonumber \lceil x\varphi_n(B_n)+ (x-1)Y_n \varphi_n(B_n)\rceil+ \varphi_n(B_n) Y_n\\ 
&\label{eq2}=S_n(x;B_1, \dots, B_n) + \varphi_n(B_n) Y_n,
\end{align}
and 
\begin{align}
&\nonumber S_n(x;B_1, \dots, B_n) + \varphi_n(B_n) Y_n=\lceil x\varphi_n(B_n)+ (x-1)Y_n \varphi_n(B_n)\rceil+ \varphi_n(B_n) Y_n\\
&\label{eq3}\leq   x\varphi_n(B_n)+ (x-1)Y_n \varphi_n(B_n)+1+ \varphi_n(B_n) Y_n =x\varphi_n(B_n)(1+Y_n)+1.
\end{align}
The result follows by combining (\ref{eq1})--(\ref{eq3}).	
\end{proof}

\noindent The bivariate extension of  Lemma \ref{udist} is presented in the result  that follows. The proof can be easily obtained by applying similar steps as in the proof of Lemma \ref{udist}  and therefore is omitted. 

\begin{lemma} 
	\label{bdist}Let $(R_n)_{n\geq 1}$ be as in (\ref{Rdef}). Then for $x,y\geq 1$ and for integers $i<j$,
	\[
	E\left[I(R_i\geq x)F_j\left(\frac{\varphi_{j}(B_{j})(1+Y_{j})}{S_{j}(y;B_1, \dots, B_{j}) + \varphi_{j}(B_{j}) Y_{j}}\right)\right]= P(R_i>x,R_j>y)\leq F_i\left(\frac{1}{x}\right)F_j\left(\frac{1}{y}\right)
	\]
	where $s_n(x;h_1, \dots, h_{n} ) := \lceil x\varphi_n(h_n)+ (x-1)y_n \varphi_n(h_n)\rceil $.
\end{lemma}

\noindent Some algebraic calculations lead to simpler and useful inequalities.  

\begin{corollary}\label{boundsforprob} Let $(R_n)_{n\geq 1}$ be as in (\ref{Rdef}). Then, for $x,y\geq 1$
	\begin{enumerate}
		\item[i.] 
		$$E\left[F_i\left(\frac{1}{x+A_i}\right)\right]\leq P(R_i>x)\leq F_{i}\left(\frac{1}{x}\right)\quad\mbox{for}\quad i=1,2,\ldots.$$
		
		\item[ii.] 
		$$E\left[F_j\left(\frac{1}{y+A_j}\right)I(R_i>x)\right] \leq P(R_i>x, R_j>y)\leq F_{i}\left(\frac{1}{x}\right) F_{j}\left(\frac{1}{y}\right)\quad\mbox{for}\quad i<j.$$
	\end{enumerate}
	where $A_j = (\varphi_j(B_j)(1+Y_j))^{-1}$ for $j=1,2,\ldots$.
\end{corollary}

\begin{proof}
	The proof is straightforward from Lemmas \ref{udist} and \ref{bdist}. 
\end{proof}

\noindent The probability inequalities described above can be simplified further  if the functions $\varphi_n$ satisfy additional conditions. 

\begin{corollary}\label{boundsforprob1} Let $(R_n)_{n\geq 1}$ be as in (\ref{Rdef}) and assume that $\varphi_n \geq 1$ for every $n$. Then for $x,y\geq 1$,
	\begin{enumerate}
		\item[i.]
		$$ F_i\left(\frac{1}{x+1}\right)\leq P(R_i>x)\leq F_{i}\left(\frac{1}{x}\right)\quad\mbox{for}\quad i=1,2,\ldots.$$
		
		\item[ii.]  
		$$F_i\left(\frac{1}{x+1}\right)F_j\left(\frac{1}{y+1}\right)\leq P(R_i>x, R_j>y)\leq F_{i}\left(\frac{1}{x}\right) F_{j}\left(\frac{1}{y}\right)\quad\mbox{for}\quad i<j.$$
	\end{enumerate}
\end{corollary}
\begin{proof}
	The result follows immediately from Corollary \ref{boundsforprob} by noticing that for the quantity $A_j$ we have that $0\leq A_j\leq 1$ for $j=1,2,\ldots$.
\end{proof}

\begin{proposition}
	\label{dependence}Let $(R_n)_{n\geq 1}$ be as in (\ref{Rdef}) with $\varphi_n\geq 1$  for every $n$.  Assume that there exists $M<\infty$ such that $\forall \, j=1,2,\dots$   
	\begin{equation}
	\label{finitesup}F_j(x)-F_j(y)\leq M(x-y) \quad\mbox{for}\quad x>y.
	\end{equation}
	Then for $i\neq j$ and $x,y\geq 1$ we have
	\[
	|P(R_i>x, R_j>y)-P(R_i>x)P(R_j>y)|\leq M\left[F_i\left(\frac{1}{x}\right)\frac{1}{y^2}+ F_j\left(\frac{1}{y}\right)\frac{1}{x^2}\right]
	\]	 
\end{proposition}
\begin{proof}
	The result follows by employing the inequalities described in Corollary \ref{boundsforprob}.
	\begin{align}&
	\nonumber P(R_i>x, R_j>y)-P(R_i>x)P(R_j>y)\leq F_i\left(\frac{1}{x}\right)F_j\left(\frac{1}{y}\right)-F_i\left(\frac{1}{x+1}\right)F_j\left(\frac{1}{y+1}\right)\\
	&=\nonumber F_i\left(\frac{1}{x}\right)\left[F_j\left(\frac{1}{y}\right)-F_j\left(\frac{1}{y+1}\right) \right]
	+\left[F_i\left(\frac{1}{x}\right)-F_i\left(\frac{1}{x+1}\right)\right]F_j\left(\frac{1}{y+1}\right)\\
	&\leq\label{upper} M\left[F_i\left(\frac{1}{x}\right)\frac{1}{y^2}+ F_j\left(\frac{1}{y}\right)\frac{1}{x^2}\right].
	\end{align}
	The reverse inequality can be obtained in a similar manner.
	\begin{align}&
	\nonumber P(R_i>x, R_j>y)-P(R_i>x)P(R_j>y)\geq F_i\left(\frac{1}{x+1}\right)F_j\left(\frac{1}{y+1}\right)-F_i\left(\frac{1}{x}\right)F_j\left(\frac{1}{y}\right)\\
	&=\nonumber F_j\left(\frac{1}{y+1}\right)\left[F_i\left(\frac{1}{x+1}\right)-F_i\left(\frac{1}{x}\right)\right]+ F_i\left(\frac{1}{x}\right)\left[F_j\left(\frac{1}{y+1}\right)-F_j\left(\frac{1}{y}\right) \right]\\
	&\geq\label{lower}- M\left[F_i\left(\frac{1}{x}\right)\frac{1}{y^2}+ F_j\left(\frac{1}{y}\right)\frac{1}{x^2}\right].
	\end{align}
	The desired result follows by combining (\ref{upper}) and (\ref{lower}).
\end{proof}

\begin{remark}
	When the corresponding densities $f_n$ exist for every $n$ and $\sup_{i,x}f_i(x)<\infty$, then $ M=  \sup_{i,x}f_i(x)$.
\end{remark}

\section{Exact Weak Laws}
In this section we provide some weak exact laws for the sequence $(R_n)_{n\geq 1}$,  i.e. the   convergence  is in probability ({\it weak}) only and the limit has a nonzero finite value ({\it exact}). The result that follows plays a significant role in the proof of the main theorem of this section.

\begin{theorem}[\cite{N2016}, Theorem 2.1] 
	\label{nakata} Let $(X_j)_{j\geq 1}$ be independent random variables whose distributions satisfy $P(|X_j|>x)\asymp x^{-\alpha}$ for $j\geq 1$ and $0<\alpha\leq 1$ and furthermore $$\limsup_{x \to \infty}\sup_{j\geq 1}x^{\alpha}P(|X_j|>x)<\infty.$$ Moreover, let $(a_n)_{n\geq 1}$ and $(b_n)_{n\geq 1}$ be positive sequences that satisfy
	\[
	\sum_{j=1}^{n} a_j^{\alpha} = o(b_n^{\alpha}).
	\]
	Then 
	\[
	\lim_{n\rightarrow \infty}b_n^{-1}\sum_{j=1}^{n}a_j\left(X_j-EX_jI\left(|X_j|\leq \frac{b_n}{a_j}\right)\right) = 0\quad\mbox{in probability}.
	\]
	In particular, if there is a constant $A$ such that
	\[
	\lim_{n\rightarrow \infty}b_n^{-1}\sum_{j=1}^{n}a_j EX_jI\left(|X_j|\leq \frac{b_n}{a_j}\right) = A
	\] 
	then
	\[
	\lim_{n\rightarrow \infty}b_n^{-1}\sum_{j=1}^{n}a_jX_j = A\quad\mbox{in probability}.
	\]	
\end{theorem}

\begin{remark}
	Theorem \ref{nakata} has been recently generalized  in \cite{MA2018} to the case of negative quadrant dependent random variables.
\end{remark}

\noindent Theorem \ref{nakata} is now used in order to obtain the main result of this section which eventually will lead to an exact weak law of large numbers. 

\begin{theorem}
	\label{WeakLaw1}Let $(R_n)_{n\geq 1}$ be as in (\ref{Rdef}). Assume that there exists $\alpha \in  (0,1]$ such that, for every $n$,
	\begin{enumerate}
		\item[i.]$F_n\asymp x^{\alpha}\quad\mbox{as}\quad x\to 0.$
		\item[ii.] {\bf Uniformity condition $\mathcal{H}_\alpha$}: 
		$$\limsup_{x \to 0}\sup_{n\geq 1} \frac{F_n(x)}{x^\alpha}< \infty.$$
	\end{enumerate}  
	Let $(a_n)_{n\geq 1}$ and $(b_n)_{n\geq 1}$ be positive sequences such that
	\begin{equation}
	\label{seqcond}\sum_{k=1}^{n}a_k^{\alpha}=o(b_n^{\alpha})\quad\mbox{as}\quad n \to \infty.
	\end{equation}
	Define $U_n$ to be a sequence of independent random variables defined on $[0,1]$ such that $U_n \sim F_n $ and let $Y_n := \frac{1}{U_n}$. 
	If there is a constant $A$ such that
	\begin{equation}
	\label{sumcond}\lim_{n\to \infty}b_n^{-1}\sum_{k=1}^n a_k EY_kI\left(Y_k\leq \frac{b_n}{a_k}\right)= A
	\end{equation}
	then \[
	\lim_{n\to \infty}b_n^{-1}\sum_{k=1}^n a_k R_k = A~~\mbox{in probability}.
	\]
\end{theorem}
\begin{proof}
	Since
	\[
	P(Y_n >x) = F_n\left(\frac{1}{x}\right), 
	\]	 
	we have 
	\begin{equation}
	\label{condition1}P(Y_n >x)\asymp x^{-\alpha}\quad\mbox{  as}\quad x\to 0.
	\end{equation}
	and
	\begin{equation}
	\label{condition2}\limsup_{x \to \infty}\sup_{n\geq 1}  x^\alpha P(Y_n >x)< \infty.
	\end{equation}
	Hence, according to Theorem \ref{nakata}, we have that
	\[
	\lim_{n\to \infty}b_n^{-1}\sum_{k=1}^n a_k Y_k = A~~\mbox{in probability}.
	\] 
	Let $W_n =b_n^{-1}\sum_{k=1}^n a_k R_k$. It is sufficient to prove that 
	\[
	\xi_{W_n}(t)\rightarrow e^{itA}
	\]
	where $\xi_{W_n}$ is the characteristic function of $W_n$. Now
	$$\xi_{W_n}(t)= E\left[e^{it\frac{1}{b_n} \sum_{k=1}^n a_kR_k } \right]
	=E\left[e^{i\sum_{k=1}^n \frac{ t a_k}{b_n} R_k} \right]
	= E\left[e^{i\sum_{k=1}^n  t_{k,n}R_k}\right]=\phi_{R_1, \dots, R_n}(t_ {{1,n} }, \dots, t_{{n,n}} )$$
	with $t_{k,n}= \frac{ ta_k }{b_n}$. By applying Theorem \ref{teoremadistanza} we have that
	$$\big| \phi_{R_1, \dots, R_n}(t_ {{1,n} }, \dots, t_{{n,n}} )-\prod_{k=1}^n\psi_k(t_ {k,n})\big|\leq  \sum_{k=1}^n |t_ {k,n}|=  |t|\sum_{k=1}^n\frac{a_k }{b_n}.$$
	Observe that for $0<\alpha\leq1$ we have that 
	$$0 \leq  \sum_{k=1}^n\frac{a_k }{b_n}\leq \left(\sum_{k=1}^n\frac{a_k^\alpha  }{b_n^\alpha}\right)^\frac{1}{ \alpha}.$$
	Thus, the desired convergence is obtained via (\ref{seqcond}).
\end{proof}

\begin{remark}
	Note that  either conditions (\ref{condition1}) or (\ref{condition2}), imply infinite mean for the random variable involved. 
\end{remark}

\begin{remark}
	It is important to highlight that the exact weak law presented in Theorem \ref{WeakLaw1} is proven without any assumptions on the dependence structure of the random variables $(R_n)_{n\geq 1}$.
\end{remark}

\noindent Theorem \ref{WeakLaw1} is the ``key" result for obtaining the  four  theorems that follow.

\begin{theorem}
	\label{WeakLaw3}   Let $(c_n)_{n\geq 1}$ be a sequence of positive numbers such that $\sup_n c_n < \infty$. Define the sequence  
	\begin{equation}
	\label{Cn}
	C_n:=\sum_{k=1}^{n}c_k^{-1},
	\end{equation}
	and assume that $$\lim_{n \to \infty}C_n=\infty.$$
	Furthermore, assume that there are real numbers $\kappa$ and $\ell$ such that the following conditions are satisfied:
	\begin{itemize}
		\item[i.]$$\lim_{n \to \infty} \frac{1}{C_n \log C_n} \sum_{k=1}^n \frac{\log c_k}{c_k }= \ell;$$
		\item[ii.]$$\lim_{n \to \infty}  \frac{n}{C_n \log C_n}= \kappa.$$
	\end{itemize} 
	Let $(R_n)_{n\geq 1}$ be as in (\ref{Rdef}) with $F_n$ given by
	$$F_n(x) = \begin{cases}0 & x < 0\\\displaystyle \frac{x}{1-c_n x} & 0 \leq x< \displaystyle\frac{1}{1+c_n}\\\\ 1 & x \geq\displaystyle \frac{1}{1+c_n}  .
	\end{cases}$$
	Then 
	\[
	\lim_{n\to\infty}\frac{1}{C_n\log C_n}\sum_{k=1}^{n}c_k^{-1}R_k= \ell+ 1+\kappa~~\mbox{in probability.}	
	\]	       
\end{theorem}

\begin{proof} 
	Note that the assumption $\sup_n c_n < \infty$ ensures that $F_n$ satisfies the uniformity condition $\mathcal{H}_1$ and that the condition $F_n \asymp x^{\alpha}$ is also satisfied with $\alpha =1$. Let $U_n$ be a sequence of independent random variables defined on $[0,1]$ such that $U_n \sim F_n $ and define $Y_n := \frac{1}{U_n}$. 
	
	\noindent Then
	$$ P(Y_n \leq y )= \begin{cases} 0 &  y < 1+ c_n\\ \displaystyle\frac{y-c_n-1}{ y-c_n} & y \geq 1+ c_n.
	\end{cases}$$
	Let $a_k = \frac{1}{c_k}$, $b_n = C_n \log C_n.$
	The sequence $Y_n$ satisfies condition \eqref{condition2}   with $\alpha =1$, so  by Theorem \ref{WeakLaw1}, it suffices to prove that
	\begin{equation*}
	\lim_{n\to \infty}b_n^{-1}\sum_{k=1}^n a_k EY_kI\left(Y_k\leq \frac{b_n}{a_k}\right)=  \ell+ 1+\kappa.
	\end{equation*}  
	We have
	\begin{equation*}
	EY_kI\left(Y_k\leq \frac{b_n}{a_k}\right)=\int_{1+ c_k}^{c_k C_n \log C_n} \frac{y}{(y-c_k)^2}\, dy = \log c_k + \log | C_n \log C_ n-1 | + c_k - \frac{  1}{ C_n \log C_n -1}.
	\end{equation*}		
	Thus,
	\begin{align*}&
	b_n^{-1}\sum_{k=1}^n a_k EY_kI\left(Y_k\leq \frac{b_n}{a_k}\right)= \frac{1}{C_n \log C_n} \sum_{k=1}^n \frac{\log c_k}{c_k }+ \frac{  \log | C_n \log C_ n-1 |}{   \log C_n}+ \frac{n}{ C_n \log C_ n}- \frac{1 }{(C_n \log C_ n-1)   \log C_ n} \\ 
	& \to \ell + 1 + \kappa   , \qquad n \to \infty.  
	\end{align*} 
\end{proof} 
\noindent   Another application of Theorem \ref{WeakLaw1} is given below by taking into consideration distribution functions of different structure. 

\begin{theorem}
	\label{WeakLaw2} Let  $(c_n)_{n\geq 1} $ be a sequence of positive numbers and let $(C_n)_{n\geq 1}$ be as in \eqref{Cn}; we assume  that
	
	\begin{itemize}
		\item[i.]$$\lim_{n \to \infty}C_n=\infty;$$
		\item[ii.]$$\lim_{n \to \infty} \frac{1}{C_n \log C_n} \sum_{k=1}^n \frac{1}{c_k }\log \left (  \frac{c_k+1}{c_k} \right  )=m\quad\mbox{for}\quad m \in \mathbb{R}.$$
		
	\end{itemize} 
	Let $(R_n)_{n\geq 1}$ be as in (\ref{Rdef}) with $F_n$ given by
	
	$$F_n(x) = \begin{cases}0 & x < 0\\\displaystyle \frac{x}{1+c_n x} & 0 \leq x\leq \displaystyle1\\\\ 1 & x >1 .
	
\end{cases}$$
Then,
\[
\lim_{n\to\infty}\frac{1}{C_n\log C_n}\sum_{k=1}^{n}c_k^{-1}R_k=  1- m  ~~\mbox{in probability,}	
\]	 
\end{theorem}
\begin{proof} The proof is similar to the preceding one. First notice that since $\inf_n c_n \geq 0$, $F_n$ satisfies both conditions of Theorem \ref{WeakLaw1} with $\alpha =1$. Let $U_n$ be a sequence of independent random variables defined on $[0,1]$ such that $U_n \sim F_n $ and define $Y_n := \frac{1}{U_n}$ and let $a_k = \frac{1}{c_k}$, $b_n = C_n \log C_n.$
	We have
	$$P(Y_n \leq y )= \begin{cases} 0 &  y < 1\\ \displaystyle\frac{y+c_n-1}{ y+c_n} & y \geq 1.
	\end{cases}$$
	By Theorem \ref{WeakLaw1}, it suffices to prove that
	\begin{equation*}
	\lim_{n\to \infty}b_n^{-1}\sum_{k=1}^n a_k EY_kI\left(Y_k\leq \frac{b_n}{a_k}\right)=   1- m.
	\end{equation*} 
	Observe that
	\begin{equation*}
	EY_kI(Y_k\leq x) = P(Y_k= 1) + \int_{1}^{x}\frac{t}{(t+c_k)^2}dt =\log\left(\frac{x+c_k}{1+c_k}\right)+  \frac{ c_k}{x+c_k}.
	\end{equation*}
	Therefore, letting again $a_k = \frac{1}{c_k}$, $b_n = C_n \log C_n,$
	\begin{align*}&
	b_n^{-1}\sum_{k=1}^n a_k EY_kI\left(Y_k\leq \frac{b_n}{a_k}\right)= \frac{1}{C_n \log C_n} \sum_{k=1}^n   \frac{1}{c_k}\left\{\log (C_n \log C_n+1) + \log \frac{c_k}{1+ c_k}+ \frac{1}{C_n \log C_n+1}\right\} \\
	&=\frac{\log(C_n \log C_n+1)}{\log C_n}  
	-\frac{1}{C_n \log C_n} \sum_{k=1}^n \frac{1}{c_k }\log \left (  \frac{c_k+1}{c_k} \right  ) + \frac{1}{(C_n \log C_n+1) \log C_n}\\ & \to   1 -  m   , \qquad n \to \infty. 
	\end{align*}
	Hence, by Theorem \ref{WeakLaw1} we have that
	\[
	\lim_{n\to\infty}\frac{\sum_{k=1}^{n}c_k^{-1}Y_k}{C_n\log C_n} =  1- m ~~\mbox{in probability.}
	\] 	
	Having established the convergence for the sequence $\{Y_n,n\geq 1\}$, the convergence of $(R_n)_{n\geq 1}$ derives by Theorem \ref{WeakLaw1}. 
\end{proof}
\begin{remark}$~$
	\begin{enumerate}
		\item[(i)]Note that the result of Theorem \ref{WeakLaw2} can be considered as a generalization of Theorem 3.1 of \cite{N2016} and of Corollary 2.2 in \cite{MA2018} since here any assumption for the  dependence structure of the random variables $(R_n)_{n\geq 1}$ is dropped. 
		
		\item[(ii)]   Observe that if we consider $c_n =1$ for every $n$  in  both  Theorems \ref{WeakLaw3} and \ref{WeakLaw2}, we obtain that 
		\[
		\lim_{n\to\infty}\frac{\sum_{k=1}^{n} R_k}{ n\log  n} = 1  ~~\mbox{in probability.}
		\] 	
		which can be also obtained from   Theorem 2.2 of \cite {G2018}. 
		
		\item[(iii)]It is easy to check that the assumptions on the sequence $c_n$ needed in Theorems \ref{WeakLaw3} and \ref{WeakLaw2} are satisfied for $c_n = \frac{1}{n^\beta}$   with $\beta \geq 0$ (with $\ell = -m=  \frac{-\beta}{\beta +1}$, $\kappa  =0$).  
	\end{enumerate}
\end{remark}

\noindent The following is an another exact weak law.

\begin{theorem}
	Let $(R_n)_{n\geq 1}$ be as in (\ref{Rdef}) with $F_n = F=$ the uniform distribution on $[0,1]$. 
	Then, for $b\geq 2$
	\[
	\lim_{n\to\infty}\frac{1}{\log^bn}\sum_{k=1}^{n}\frac{\log^{b-2}k}{k}R_k= \frac{1}{b}~~\mbox{in probability.}
	\]
\end{theorem}
\begin{proof}
	Let $U_n$ be independent and uniformly distributed random variables on $[0,1]$ and $Y_n= \frac{1}{U_n}$. Then
	\[
	P(Y_k\leq y)  = \begin{cases} 0 & y <1\\ 1-\displaystyle \frac{1}{y}& y \geq 1.\end{cases} 
	\]
	Let $b_n = \log^bn$ and $a_k = \frac{\log^{b-2}k}{k}$ for $b\geq 2$. Observe that 
	\[
	\lim_{n \to \infty}\frac{1}{\log^bn}\sum_{k=1}^n\frac{\log^{b-2}k}{k} = 0,
	\] 
	i.e condition (\ref{seqcond}) is satisfied with $\alpha =1$. Note that 
	\begin{align*}
	E\left(Y_kI\left(Y_k\leq \frac{b_n}{a_k} \right)\right) =\int_1^\frac{b_n}{a_k}  \frac{1}{t} dt = \log \left(\frac{b_n}{a_k}\right)= \log k + b \log \log n - (b-2) \log \log k.
	\end{align*} 
	Thus,
	\begin{align*}&
	b_n^{-1}\sum_{k=1}^na_kE\left(Y_kI\left(Y_k\leq \frac{b_n}{a_k} \right)\right) 
	=\frac{1}{ \log^bn}\sum_{k=1}^n\frac{\log^{b-1}k}{k}+  \frac{ b \log \log n}{ \log^bn}\sum_{k=1}^n\frac{\log^{b-2}k}{k}\\&~~~~~~~~~~~~~~~~~~~~~~~~~~~~~~~~~~~~~~~~~~~~~~~-\frac{ b-2  }{ \log^bn}\sum_{k=1}^n\frac{\log^{b-2}k\cdot\log \log k}{k}  \to \frac{1}{b}, \quad n \to \infty.
	\end{align*}	 
	The convergence of the sequence $\{R_n,n\geq 1\}$ is established by  Theorem \ref{WeakLaw1}.
\end{proof}

\begin{remark}
	In both papers \cite{N2016} and \cite{MA2018} (see Corollary 2.1 in both), it is proven that in the case where $\alpha \in (0,1)$ the limit of the weighted partial sum is equal to zero, i.e. the weak law is established but it is not an exact weak law. It is of interest to check whether this result is also valid in this framework as well. The answer is given by the result that follows.
\end{remark}

\begin{theorem}
	Let $(R_n)_{n\geq 1}$ be as in (\ref{Rdef}) with $F_n(x) = x^\alpha$ on $[0,1]$ and $0<\alpha <1$ for every $n$. For every $(a_n)_{n\geq 1}$ and $(b_n)_{n\geq 1}$ such that condition (\ref{seqcond}) holds, 
	we have 
	\[
	\lim_{n\to\infty}\frac{1}{b_n}\sum_{k=1}^{n}a_kR_k=   0~~\mbox{in probability}.
	\] 
\end{theorem}

\begin{proof}
	By applying similar steps as in the previous proofs we have that
	\begin{align*}
	b_n^{-1}\sum_{k=1}^na_kE\left(Y_kI\left(Y_k\leq \frac{b_n}{a_k} \right)\right) =\frac{\alpha}{1-\alpha} \sum_{k=1}^n\frac{a_k }{ b_n}\left[ \left(\frac{b_n}{a_k}\right)^{1-\alpha} -1\right]= 
	\frac{\alpha}{1-\alpha} \left[\sum_{k=1}^n\Big(\frac{a_k }{ b_n}\Big)^\alpha -\sum_{k=1}^n\left(\frac{a_k }{ b_n}\right)\right]\to 0.
	\end{align*}		
\end{proof}

\section{Exact Strong Laws} 
We start by proving  a result on the behaviour of the tails of $R_n$.
\begin{theorem} 
	\label{uniftails} Let $(R_n)_{n\geq 1}$ be as in (\ref{Rdef}) with  $F_n$ for which  there exist $\alpha >0$ and  $c>0$ such that
	
	\begin{equation}\label{uniformita}
	\lim_{t \to 0}\sup_n \left|\frac{F_n(t)}{t^\alpha}- c\right|=0.
	\end{equation}
	\begin{enumerate}
		\item[i.]Let $U_n\sim F_n$ and define $Y_n= \frac{1}{U_n}$ for every $n$.  Then
		\[
		\lim_{x\rightarrow\infty}\sup_{n}\left|\frac{P(R_n>x)}{P(Y_n>x)}-1\right|=0.
		\]
		\item[ii.]For every fixed $m$ \[\lim_{x\rightarrow\infty} \sup_{n}\left|\frac{P(R_n>x)}{P(R_m>x)}-1\right|=0.\]
	\end{enumerate} 
\end{theorem}
\begin{proof}
	The first part can be easily derived by using the inequalities described in Corollary \ref{boundsforprob1} i.e.
	$$\frac{ F_n \left(\frac{1}{x+ 1} \right) }{\left(\frac{1}{x+ 1}\right)^\alpha}\cdot \frac{x^\alpha}{(x+1)^\alpha}\cdot \frac{\left( \frac{1}{x}\right)^\alpha}{ F_n\left(\frac{1}{x}\right)}= \frac{  F_n\left(\frac{1}{x+ 1}\right) }{ F_n(\frac{1}{x})}  \leq   \frac{P(R_n>x)}{ P( Y_n> x)} \leq 1,$$
	and the result follows immediately from \eqref{uniformita}. 
	
	\noindent The second part of the Theorem follows easily since, by the first part it suffices to prove the same relation with $Y_n$ and $Y_m$ in place of $R_n$ and $R_m$ respectively. Then,
	
	$$\frac{P(Y_n>x)}{ P( Y_m> x)} = \frac{F_n\left(\frac{1}{x}\right)}{ F_m\left(\frac{1}{x}\right)}= \frac{F_n\left(\frac{1}{x}\right)}{ \left( \frac{1}{x}\right)^\alpha}\cdot \frac{ \left(\frac{1}{x}\right)^\alpha}{ F_m\left(\frac{1}{x}\right)}$$
	Let $\epsilon \in ( 0, c)$ be fixed. By assumption \eqref{uniformita} there exists $ \delta \in (0,1)$ such that, for every $t \in (0,\delta)$ we have
	$$ c-\epsilon   <\frac{ F_n(t)}{ t^\alpha}  <  c+\epsilon$$
	for every $n$. Therefore, for sufficiently large $x$
	$$\frac{c-\epsilon}{ c+\epsilon} < \frac{P(Y_n>x)}{ P( Y_m> x)} <\frac{c+\epsilon}{ c-\epsilon}.$$
	The desired result follows by the arbitrariness of $\epsilon$. 
\end{proof}

\begin{remark}
	The above result indicates  that the sequence $(R_n)_{n\geq 1}$ has uniformly equivalent tails to the tails of the random variable $Y_n$ and to the tails of every $R_m$ with $m$ fixed.
\end{remark}

\begin{remark}
	Assume that $F_n$ has a density $f_n$ for every $n$. Then a sufficient condition for \eqref{uniformita} is that there exist $\alpha >0$ and  $c>0$ such that
	
	$$\lim_{t \to 0}\sup_n \left|\frac{f_n(t)}{  t^{\alpha-1}}-   c\right|=0.$$
	Note that the latter is a generalization of the condition used in Theorems 2.2 and 2.3 of \cite{G2018}.
\end{remark}

\subsection{A strong law for the independence case}
It is important to mention that for the sequence $(R_n)_{n\geq 1}$ no dependence structure is assumed as this may vary depending on $\varphi_n$ and the choice of $y_n$. However, the result that follows provides a special case where the random variables $R_n$ are independent.

\begin{proposition}
	\label{indRn}Let $(R_n)_{n\geq 1}$ be as defined in (\ref{Rdef}) with $\varphi_n(h_n) = c_n,~\forall h$ and $y_n = y_n(h_1,\ldots,h_n) = d_n~\forall h_1,\ldots,h_n$. Then, the sequence $(R_n)_{n\geq 1}$ consists of independent random variables. 
\end{proposition}
\begin{proof}
	By Lemma \ref{bdist} we have that
	\begin{align*}
	&P(R_n > x, R_{n+1}>y)-P(R_n > x)P( R_{n+1}>y)\\
	&= E\left[I(R_n\geq x)F_{n+1}\left(\frac{\phi_{n+1}(B_{n+1})(1+Y_{n+1})}{S_{n+1}(y;B_1, \dots, B_{n+1}) + \phi_{n+1}(B_{n+1}) Y_{n+1}}\right)\right]\\
	&~~-E\left[I(R_n\geq x)\right]E\left[ F_{n+1}\left(\frac{\phi_{n+1}(B_{n+1})(1+Y_{n+1})}{S_{n+1}(y;B_1, \dots, B_{n+1}) + \phi_{n+1}(B_{n+1}) Y_{n+1}}\right)\right].
	\end{align*}
	Put for simplicity
	$$ F_n\left(\frac{\phi_{n}(B_{n})(1+Y_{n})}{S_{n}(u;B_1, \dots, B_{n}) + \phi_{n}(B_{n}) Y_{n}}\right)=:Z^{(u)}_n. $$
	Then
	\begin{eqnarray*}
		P(R_n > x, R_{n+1}>y)-P(R_n > x)P( R_{n+1}>y)&=&E\big[I(R_n\geq x)Z^{(y)}_{n+1}\big]-E\big[1_{\{R_n\geq x\}}\big]E\big[ Z^{(y)}_{n+1}\big]\\
		&=&E\big[I(R_n\geq x)\big\{Z^{(y)}_{n+1}-E\big[ Z^{(y)}_{n+1}\big]\big\}\big].
	\end{eqnarray*}
	For the particular choices of  $\varphi_n$ and $y_n$ we have that 
	$$S_{n}(u;B_1, \dots, B_{n})= \lceil u\phi_n(B_n)+ (u-1)Y_n(B_1, \dots, B_n) \phi_n(B_n)\rceil = \lceil uc_n + (u-1)d_n c_n\rceil$$ therefore for every $\omega \in \Omega$
	$$Z^{(u)}_{n}(\omega)= F_n\left(\frac{c_n(1+d_n)}{ \lceil uc_n + (u-1)d_n c_n\rceil + c_n d_n }\right),$$
	i.e. $\omega \mapsto Z^{(u)}_{n}(\omega)$ is constant (in $\omega$), leading to
	$$Z^{(y)}_{n+1}-E\big[ Z^{(y)}_{n+1}\big]=0.$$
	As a consequence
	$$P(R_n > x, R_{n+1}>y)-P(R_n > x)P( R_{n+1}>y)=E\big[I(R_n\geq x)\big\{Z^{(y)}_{n+1}-E\big[ Z^{(y)}_{n+1}\big]\big\}\big]=0.$$
	By the same argument we can prove that in general for $x_i\geq 1,~i=0,1,\ldots,k$
	\[
	P(R_n>x_0, R_{n+1}>x_1,\ldots,R_{n+k}>x_k) = E(Z_{n+k}^{(x_{k})}I(R_n>x_0,\ldots,R_{n+k-1}>x_{k-1})).
	\]
	Therefore
	\begin{eqnarray*}
		&~&P(R_n>x_0, R_{n+1}>x_1,\ldots,R_{n+k}>x_k)-\prod_{i=0}^{k-1}P(R_{n+i}>x_i)P(R_{n+k}>x_k)\\
		&=&E(Z_{n+k}^{(x_{k})}I(R_n>x_0,\ldots,R_{n+k-1}>x_{k-1}))-EZ_{n+k}^{(x_{k})}\prod_{i=0}^{k-1}P(R_{n+i}>x_i)\\
		&=&Z_{n+k}^{(x_{k})}\left(P(R_n>x_0, R_{n+1}>x_1,\ldots,R_{n+k-1}>x_{k-1})-\prod_{i=0}^{k-1}P(R_{n+i}>x_i)\right)
	\end{eqnarray*}
	where the last equality is derived due to the fact that $Z_n^{(u)}$ is constant with respect to $\omega$. Continuing this pattern we will have 
	\begin{eqnarray*}
		&~&P(R_n>x_0, R_{n+1}>x_1,\ldots,R_{n+k}>x_k)-\prod_{i=0}^{k-1}P(R_{n+i}>x_i)P(R_{n+k}>x_k)\\
		&=&Z_{n+k}^{(x_{k})}\cdots Z_{n+2}^{(x_{2})}\left(P(R_n>x_0,R_{n+1}>x_1)-P(R_n>x_0)P(R_{n+1}>x_{1})\right)\\
		&=&0
	\end{eqnarray*}
	i.e. independence is established.
\end{proof}

\begin{remark}
	Note that the result presented above requires no assumptions for the distribution functions $F_n$. In the special case where $F_n=$uniform distribution on $[0,1]$, $\varphi_n\equiv 1$ and $y_n \equiv 0$ for every $n$, the construction of $R_n$ reduces to the well-known case of the Luroth series \cite {L1883} for which independence is known (see for example \cite {G1976}). 
\end{remark}

\noindent The exact strong law that follows is a direct consequence of Theorem 4.1 of \cite{AM2018}.

\begin{theorem}
	\label{exactstrind}Let $(R_n)_{n\geq 1}$ be as in Proposition \ref{indRn} and assume that the distribution functions $F_n$ satisfy \eqref{uniformita} with $\alpha =1$. Then for every $b>2$,
	\[
	\lim_{n\rightarrow \infty}\frac{1}{\log^bn}\sum_{k=1}^{n}\frac{\log^{b-2}k}{ k}R_k=\frac{1}{b}~~a.s.
	\]
\end{theorem}
\begin{proof}
	For the proof, we check that the assumptions of Theorem 4.1 in \cite{AM2018} are satisfied. Since $(R_n)_{n\geq 1}$ is a sequence of independent random variables, Assumption (1.2) is satisfied. Assumption (1.3) is satisfied by Theorem \ref{uniftails} (b). Let $m$ be the integer considered in Theorem \ref{uniftails} (b); then the assumption \eqref{uniformita}
	ensures that
	\[
	L(x) = \begin{cases}
	x  F_m(1/x) & x\geq 1\\
	x  & x<1
	\end{cases}
	\] 
	is a slowly varying function and therefore the expression (3.2) in \cite{AM2018} is also verified. Following the notation of \cite{AM2018} the sequence $c_n$ is defined as 
	\[
	c_n = n\left(\int_{1}^{c_n}P(Y_m>t){\rm d}t\right)\log(c_n+e).
	\]
	Observe that by condition \eqref{uniformita} $EY_m = \infty$ and by employing condition \eqref{uniformita} again  it can be easily verified that
	\[
	\int_{1}^{c_n}P(Y_m>t){\rm d}t\sim \log c_n, 
	\]
	which leads to the conclusion that $c_n \sim n \log^2 n$ (since $c_n$ goes to $\infty$, as remarked in \cite{AM2018}, p. 109). Then
	$$\sum_n P(R_m>c_n)\leq \sum_n F_m\left(\frac{1}{c_n}\right)\sim \sum_{n} \frac{1}{n \log ^2 n} <\infty$$ 
	where the first inequality follows by Lemma \ref{udist} while the equivalence is obtained by condition \eqref{uniformita}. Thus, the result follows immediately.
\end{proof}

\begin{remark} 
	It is still an open question to find more general conditions than independence (if any) under which  the result of Theorem \ref{exactstrind} holds. 
\end{remark}

\subsection{A strong law in the general case}
Throughout this section, $c_n = n\log^bn$ for $b\geq 2$ and the sequence of functions denoted by $g_n:[1, + \infty)\rightarrow \mathbb{R}$ will be of the form
\begin{equation*}
g_n(x) =-c_nI(x<-c_n)+x I(|x|\leq c_n)+c_n I (x>c_n).
\end{equation*}

\noindent Before stating the main result of this section, we first present some useful lemmas. As it has already been mentioned in the introduction, the symbol $C$ appearing throughout may represent different constant every time.

\begin{lemma}
	\label{cov}Let $(R_n)_{n\geq 1}$ be as in (\ref{Rdef}) with $\varphi_n\geq 1$ and $F_n$ such that conditions (\ref{finitesup}) and (\ref{uniformita}) for $\alpha =1$ are satisfied. Define $W_n = \frac{1}{n}g_n(R_n)$. Then for $i \neq j$
	\[\left|\Cov (W_i, W_j)\right|  \leq \frac{C}{ij}\left(\log i+ \log j\right),
	\]
	where $C$ is a positive constant.
\end{lemma}
\begin{proof}
	Note that
	\[
	\Cov (W_i, W_j) = \frac{1}{ij}Cov (g_i(R_i),g_j(R_j)).
	\]
	First, consider the case where $c_i \geq 1$ for all $i$. By using the definition of the sequence $c_n$ we have that
	\begin{equation}
	\label{transf}P\big(g_i(R_i)>u, g_j(R_j)>v\big)= \begin{cases}
	P(R_i>u, R_j >v) & u < c_i, v< c_j\\ 0 & \rm otherwise.
	\end{cases}
	\end{equation}
	It is known that for any positive random variables $X$ and $Y$ (not necessarily absolutely continuous)
	\[
	E[X]= \int_0^\infty P(X>x) {\rm d} x; \qquad  E[XY]= \int_0^\infty\int_0^\infty P(X>x, Y>y) {\rm d} x{\rm d} y
	\]
	and by observing that $R_n\geq 1$ we have that
	\begin{align*}&
	E[g_i(R_i)g_j(R_j)]= \int_0^\infty\int_0^\infty P( g_i(R_i)>x,  g_j(R_j)>y){\rm d} x{\rm d} y\\&= \int_0^1{\rm d} x\int_0^1 1 {\rm d} y+ \int_0^1{\rm d} x\int_1^\infty {\rm d} y P(g_j(R_j)>y)
	+ \int_1^\infty{\rm d} x \int_0^1{\rm d} yP(g_i(R_i)>x)\\&+ \int_1^\infty{\rm d} x \int_1^\infty{\rm d} yP( g_i(R_i)>x,  g_j(R_j)>y)\\&= 1+ \int_1^{c_j}P(R_j > y){\rm d} y
	+ \int_1^{c_i}P(R_i > x){\rm d} x+ \int_1^ {c_i}{\rm d} x \int_1^{c_j}{\rm d} yP(R_i > x, R_j >y),\end{align*}
	where in the last equality we have used the expression obtained in (\ref{transf}). Similarly,
	\begin{align*}
	& E[g_i(R_i)]E[g_j(R_j)] \\&= 1+ \int_1^{c_j}P(R_j > y){\rm d} y
	+ \int_1^{c_i}P(R_i > x){\rm d} x+ \int_1^ {c_i}{\rm d} x \int_1^{c_j}{\rm d} yP(R_i > x)P( R_j >y).
	\end{align*}
	Thus,
	\begin{equation*}\label{intermedia}
	\Cov(g_i(R_i),g_j(R_j))=  \int_1^ {c_i}{\rm d} x \int_1^{c_j}{\rm d} y \big\{P(R_i > x, R_j >y)-P(R_i > x)P( R_j >y)\big\}.
	\end{equation*}
	By applying Proposition \ref{dependence} we have that 
	\begin{align*} 
	&\big|\Cov(g_i(R_i),g_j(R_j))\big| \leq M\int_1^ {c_i}{\rm d} x \int_1^{c_j}{\rm d} y \left[
	F_i\left(\frac{1}{x}\right) \frac{1}{y^2}+ F_j\left(\frac{1}{y}\right) \frac{1}{x^2}\right]\\ 
	& =  M\left[\int_1^ {c_i}{\rm d} xF_i\left(\frac{1}{x}\right)\int_1^{c_j}{\rm d} y \frac{1}{y^2}+\int_1^ {c_i}{\rm d} x\frac{1}{x^2}\int_1^{c_j}{\rm d} y  F_j\left(\frac{1}{y}\right)
	\right].
	\end{align*}
	A change of variable leads to
	\[
	\int_1^ {c_i}{\rm d} xF_i\left(\frac{1}{x}\right)= 
	\int_\frac{1}{c_i}^ 1 \frac{F_i(t)}{t^2} {\rm d}t
	\sim c \log i, \qquad \hbox{as } i \to \infty.
	\]
	The last equivalence is proven as follows. By using condition \eqref{uniformita} and for fixed $\epsilon >0$, let $\delta\in (0,1)$ be such that 
	$$c- \epsilon\leq \frac{F_i(t)}{t} < c+ \epsilon$$ for $0<t<\delta$, and let $i_0$ be sufficiently large in order that $\frac{1}{c_i}< \delta$ for every $i>i_0$. Then
	$$\int_\frac{1}{c_i}^\delta \frac{c-\epsilon}{t} {\rm d}t< \int_\frac{1}{c_i}^\delta \frac{F_i(t)}{t^2} {\rm d}t < \int_\frac{1}{c_i}^\delta \frac{c+\epsilon}{t} {\rm d} t,$$
	which amounts to
	$$(c-\epsilon) \log \delta +(c-\epsilon) \log c_i < \int_\frac{1}{c_i}^\delta \frac{F_i(t)}{t^2} {\rm d}t < (c+\epsilon) \log \delta +(c+\epsilon) \log c_i,$$ 
	where the arbitrariness of $\epsilon$ implies that
	$$\int_\frac{1}{c_i}^\delta \frac{F_i(t)}{t^2} {\rm d}t \sim c \log c_i, \qquad i \to \infty.$$
	Since
	$$\frac{ \int_ \delta^1 \frac{F_i(t)}{t^2}}{ \log c_i}\leq \frac{ C\int_ \delta^1    \frac{1}{t} {\rm d} t}{ \log c_i}\to 0, \qquad i \to \infty,$$
	we conclude that
	$$\int_\frac{1}{c_i}^1 \frac{F_i(t)}{t^2}\sim c \log c_i\sim c \log i, \qquad i \to \infty. $$ 
	Observe that
	\begin{align*}&\int_1^{c_j}{\rm d} y \frac{1}{y^2}=  1- \frac{1}{ c_j} \leq 1. \end{align*}
	Therefore
	\begin{align*} &\big|\Cov(g_i(R_i),g_j(R_j))\big| \leq C\big(\log i+ \log j\big),\end{align*}
	and as a consequence
	$$\big|\Cov (W_i, W_j)\big| = \frac{1}{ij}\big|\Cov (g_i(R_i),g_j(R_j))\big|\leq \frac{C}{ij}\big(\log i+ \log j\big),  $$ as claimed. In the case where $c_i<1$ for some $i$ it can be easily proven that $\big|\Cov (W_i, W_j)\big|=0$ for every $j$, which again is compatible with the desired result.
\end{proof}

\begin{lemma}
	\label{varbound}Under the assumptions of Lemma \ref{cov} and for $j=1,2,\ldots$
	\[
	\Var g_j(R_j)\leq Cc_j
	\]
	where $C$ is a positive constant.
\end{lemma}
\begin{proof}
	For $c_j\geq 1$, by Lemma \ref{udist},
	\begin{align*}&
	\Var g_j(R_j)\leq E[g_j^2 (R_j)]= \int_0^\infty P(g_j^2(R_j) > x){\rm d} x = \int_0^1 1 {\rm d} x+ \int_1^{c_j^2}  P(R^2_j> x) {\rm d} x\\
	&=   1 + \int_1^{c_j^2}  P(R _j> \sqrt x) {\rm d} x=   
	1 + \int_1^{c_j } 2t P(R _j> t) {\rm d} t\leq 1 + \int_1^{c_j } 2t F_j\left(\frac{1}{t}\right) {\rm d} t \leq Cc_j
	\end{align*}
	where the last inequality follows because of condition \eqref{uniformita}. It can easily be proven that for the cases where $c_j<1$ for some $j$, the statement is still valid as $\Var g_j(R_j)\leq c_j$.
\end{proof}

\begin{lemma}
	\label{boundsVarCov} Under the assumptions of Lemma \ref{cov}, 
	\begin{itemize}
		\item[i.] $$\sum_{j=1}^n \Var W_j \leq C \log^{b+1} n,~b\geq 2 \mbox{ for } j=1,2,\ldots,$$
		\item[ii.] $$\sum_{1 \leq i< j\leq n}\big|\Cov (W_i, W_j)\big| \leq C \log^3 n \mbox{ for } i\neq j.$$
	\end{itemize}
\end{lemma}
\begin{proof}
	The first inequality can be easily derived from Lemma \ref{varbound}. In detail, 
	\begin{align*}
	\sum_{j=1}^n \Var W_j \leq  C  \sum_{j=1}^n \frac{1}{j^2}c_j= C  \sum_{j=1}^n \frac{1}{j }\log^b j< C\log^{b+1}n, \qquad n \to \infty,
	\end{align*}
	where the last equivalence follows from Cesaro Theorem.
	
	\noindent The key result for obtaining the second inequality is Lemma \ref{cov} i.e.
	\[
	\sum_{1 \leq i< j\leq n}\big|\Cov (W_i, W_j)\big| \leq  C\sum_{j=2}^n \sum_{i=1}^{j-1} \frac{1}{ij}\big(\log i+ \log j\big)\sim    C\log^3 n , \qquad n \to \infty,
	\]
	where again the last equivalence follows from Cesaro Theorem (applied twice).
\end{proof}

\noindent The result that follows is instrumental for obtaining a strong law of large numbers. 

\begin{theorem}
	\label{A1}Under the conditions of Lemma \ref{cov} and for $d_n = n^\gamma$ with $\gamma >\frac{1}{2}$,
	\[
	\lim_n \frac{1}{d_n}\sum_{k=1}^n\Big\{\frac{1}{k}\big(g_k(R_k)-E[g_k(R_k)]\big)\Big\}= 0\quad \mbox{a.s.}.
	\]
\end{theorem}
\begin{proof}
	Let $S_n = \sum_{j=1}^{n}W_j = \sum_{j=1}^{n}\frac {1}{j} g_j(R_j)$. It is sufficient to prove that for every $\epsilon>0$,
	\begin{equation}
	\label{completeconv}\sum_{n}P\left(\frac{1}{d_n}|S_n-ES_n|>\epsilon\right)<\infty.
	\end{equation}
	Then the desired result follows immediately by applying the Borel-Cantelli lemma. By Chebychev inequality 
	\[
	\sum_{n}P\left(\frac{1}{d_n}|S_n-ES_n|>\epsilon\right)\leq \frac{1}{\epsilon^2}\sum_{n}\frac{\Var S_n}{d_n^2},
	\]
	so it is sufficient to prove that
	\[
	\sum_{n}\frac{\Var S_n}{d_n^2}<\infty.
	\]
	Observe that by Lemma \ref{boundsVarCov}
	\begin{align*}
	\Var S_{n } = \sum_{j=1}^ {n } \Var W_j + 2 \sum_{1 \leq i < j \leq n } \Cov (W_i, W_j) \leq C \log^{b+1} n+C\log^3 n
	\end{align*}
	where $C$ are positive constants.
	Hence
	$$\sum_n \frac{\Var S_{n  }}{d^2_{n }}\leq \sum_n \frac{ C\log^{b+1}  n}{n^{2\gamma}}+\sum_n \frac{ C\log^3  n}{n^{2\gamma}}  < \infty.$$
\end{proof}

\begin{remark}
	It is important to mention that the result described above also proves complete convergence for the sequence $\{S_n,n\geq 1\}$ due to (\ref{completeconv}).
\end{remark}

\noindent The main result of the section is presented below. 

\begin{theorem}
	\label{genstronglaw}Let $(R_n)_{n\geq 1}$ be as in (\ref{Rdef}) with   $F_n$ satisfying conditions \eqref{finitesup} and \eqref{uniformita} for $\alpha =1$. Then, for $d_n =n^{\gamma}$ with $\gamma>1$,
	\[
	\frac{1}{d_n}\sum_{k=1}^{n}\frac{R_k}{k} \rightarrow 0,~a.s.
	\]
\end{theorem}
\begin{proof}
	First for a fixed integer $m$, define the random variable $Y_m$ to be $Y_m:= \frac{1}{U_m}$ where $U_m\sim F_m(x)$. Therefore, by \eqref{uniformita}
	\begin{equation}
	\label{cond1}\sum_n P(Y_m>c_n) = \sum_{n}F_m\left(\frac{1}{c_n}\right) <\infty.
	\end{equation}
	Motivated by the proof of Theorem 4.1 of \cite{AM2018} we can write
	\begin{eqnarray*}
		\frac{1}{d_n}\sum_{k=1}^{n}\frac{R_k}{k}&=&\frac{1}{d_n}\sum_{k=1}^{n}\frac{1}{k}(g_k(R_k)-Eg_k(R_k))\\
		&+&\frac{1}{d_n}\sum_{k=1}^{n}\frac{R_k}{k}I(R_k>c_k)+\frac{1}{d_n}\sum_{k=1}^{n}\frac{c_k}{k}I(R_k<-c_k)-\frac{1}{d_n}\sum_{k=1}^{n}\frac{c_k}{k}I(R_k>c_k)\\
		&+&\frac{1}{d_n}\sum_{k=1}^{n}\frac{c_k}{k}P(R_k>c_k)-\frac{1}{d_n}\sum_{k=1}^{n}\frac{c_k}{k}I(R_k<-c_k)\\
		&+&\frac{1}{d_n}\sum_{k=1}^{n}\frac{1}{k}ER_kI(R_k\leq c_k)\\
		&:=&A_1+A_2+A_3+A_4.
	\end{eqnarray*}
	By Theorem \ref{A1}, $A_1$ tends to zero almost surely. By Lemma 3.4 of \cite{AM2018} and since (\ref{cond1}) is satisfied, $\sum_{n}P(R_n>c_n)<\infty$. Then, the first Borel-Cantelli Lemma ensures that $A_2\rightarrow 0$ almost surely as $n\rightarrow\infty$. Condition (\ref{cond1}) and Kronecker's lemma lead to $A_3 \rightarrow 0$ almost surely. By Lemma 4.5 of \cite{AM2018} we have that 
	\[
	\lim_{n\rightarrow \infty}\frac{ER_nI(R_n\leq c_n)}{\mu(c_n)} = 1
	\] 
	where $\mu(x) = \int_{1}^{x}P(Y_m>t)dt$. Thus (see \cite{A2000} p. 148)
	\[
	\frac{1}{d_n}\sum_{k=1}^{n}\frac{R_k}{k}\sim\frac{1}{d_n}\sum_{k=1}^{n}\frac{\mu(c_k)}{k}.
	\]
	Observe that
	\[
	\mu(c_k) = \int_{1}^{c_k}P(Y_m>t) dt\leq (c_k-1)<c_k = k\log^b k.
	\]
	Then 
	\[
	0<\frac{1}{d_n}\sum_{k=1}^{n}\frac{\mu(c_k)}{k} <\frac{1}{n^{\gamma}}\sum_{k=1}^{n}\log^bk\sim\frac{1}{n^{\gamma}}\int_{1}^{n}\log^bx {\rm d }{x}\rightarrow 0,
	\]
	which completes the proof.
\end{proof}

\begin{remark}
	Observe that Theorem \ref{genstronglaw} is proven under no assumption on the dependence structure of $R_n$. As already remarked,  finding  more general conditions than independence   under which  the result of Theorem \ref{exactstrind} holds is an open problem. In order to motivate the above result, we notice that Theorem \ref{genstronglaw} is a partial confirmation in this direction, since $\log^b n = o(n^\gamma)$.
\end{remark}

\begin{remark}
	It is important to be pointed out that Theorem \ref{genstronglaw} cannot be considered as an exact law, since the weighted sum involved converges to 0.
\end{remark}

\begin{remark}
	As it has been pointed out to us by the referee, quite often
	there happens to be complete convergence whenever we have almost sure
	convergence. Thus, it would be of interest to check whether the
	exact strong laws obtained in this paper can be generalized to complete
	exact laws similar to the ones studied in \cite{A2000b}.
\end{remark}

\subsection*{Acknowledgements}
R. Giuliano wishes to thank Prof. Tasos Christofides for his kind invitation at the University of Cyprus; the present paper was started during the permanence there.

\end{document}